\documentclass[12pt]{amsart}
\usepackage{amsmath,amssymb,amsthm,mathrsfs,color,dsfont}
 \usepackage{graphicx}
 \usepackage{float}
\newtheorem{thm}{Theorem}
\newtheorem{thms}{Theorem}[section]
\newtheorem{lem}{Lemma}[section]
\newtheorem{prop}{Proposition}[section]
\newtheorem{cor}{Corollary}[section]

\numberwithin{equation}{section}
\newtheorem{defi}{Definition}

\theoremstyle{definition}

\theoremstyle{remark}

\newcommand{\C}{{\mathbb C}}

\newcommand{\R}{{\mathbb R}}

\renewcommand{\Re}{\mathrm{Re}}
\renewcommand{\Im}{\mathrm{Im}}
\def\11{{\rm 1~\hspace{-1.4ex}l} }
\newcommand{\Arg}{\mathrm{Arg}}
\newcommand{\Ker}{\;\mathrm{Ker}}

\newcommand{\vertiii}[1]{{\left\vert\kern-0.25ex\left\vert\kern-0.25ex\left\vert #1
    \right\vert\kern-0.25ex\right\vert\kern-0.25ex\right\vert}}

\begin{document}
\title[Orbital stability for the generalized Choquard model]{Orbital stability of solitary waves for the generalized Choquard model }

\author{Vladimir Georgiev}
\address{V. Georgiev,  Dipartimento di Matematica, Universit\`a di Pisa
Largo B. Pontecorvo 5, 56127 Pisa, Italy \\
and \\
 Faculty of Science and Engineering \\ Waseda University \\
 3-4-1, Okubo, Shinjuku-ku, Tokyo 169-8555 \\
Japan, and IMI--BAS, Acad. Georgi Bonchev Str., Block 8, 1113 Sofia, Bulgaria}
\email{georgiev@dm.unipi.it}

\author{Mirko Tarulli}
\address{M. Tarulli, Faculty of Applied Mathematics and Informatics, Technical University of Sofia, Kliment Ohridski Blvd. 8, 1000 Sofia and IMI--BAS, Acad. Georgi Bonchev Str., Block 8, 1113 Sofia, Bulgaria\\
Dipartimento di Matematica, Universit\`a di Pisa
Largo Bruno Pontecorvo 5 I - 56127 Pisa. Italy.}
\email{mta@tu-sofia.bg}

\author{George Venkov}
\address{G. Venkov, Faculty of Applied Mathematics and Informatics, Technical University of Sofia, Kliment Ohridski Blvd. 8, 1000 Sofia, Bulgaria}
\email{gvenkov@tu-sofia.bg}%

\thanks{ The first  author   was supported in part by  INDAM, GNAMPA - Gruppo Nazionale per l'Analisi Matematica, la Probabilita e le loro Applicazioni, by Institute of Mathematics and Informatics, Bulgarian Academy of Sciences, by Top Global University Project, Waseda University,  the Project PRA 2018 49 of  University of Pisa and project "Dinamica di equazioni nonlineari dispersive", "Fondazione di Sardegna" , 2016}

\begin{abstract}
We consider the generalized Choquard equation describing trapped electron gas in 3 dimensional case.
The study of orbital stability of the energy minimizers (known as ground states) depends essentially in the local uniqueness of these minimizers.
In equivalent way one can optimize the Gagliardo--Nirenberg inequality subject to the constraint fixing the $L^2$ norm. The uniqueness of the minimizers for the case $p=2$, i.e. for the case of Hartree--Choquard is well known. The main difficulty for the case $p > 2$ is connected with possible lack of control on the $L^p$ norm of the minimizers.
\end{abstract}

\renewcommand{\thefootnote}{\fnsymbol{footnote}}
\footnotetext{\emph{Key words:} generalized Choquard equation, local uniqueness, ground states}
\footnotetext{\emph{AMS Subject Classifications:} 37K40, 35Q55, 35Q51}

\maketitle

\section{Main results}
The active study of the existence and qualitative behaviour of standing waves is motivated by the important question of stability/instability properties of these waves. Therefore, one has to justify the $H^1$-evolution dynamics of the corresponding Cauchy problem
 \begin{eqnarray}\label{eq.Hart}
 & & i \partial_t u+\Delta u + I(|u|^p) |u|^{p-2} u=0, \quad (t,x)\in
    \R_{+}\times\R^3,
 \\ \nonumber
& & u(0,x)=u_0(x)
\end{eqnarray}
and then to approach  orbital stability/instability problem. Here and below  $I(f)$ is the  Riesz potential defined by
\begin{equation} \label{eq.Dsfg2}
I(f)(x) =(-\Delta)^{-1}f(x)= \frac{1}{4\pi}\int_{\R^3} \frac{f(y)dy}{|x-y|^{}}. \quad
\end{equation}

In general the existence of ground state is studied in \cite{GV}, \cite{MVS1}, \cite{MVS2} and decay and scattering properties in  \cite{TV}. A detailed classification result for  linearized stability properties of the standing waves is obtained in \cite{GS}. Considering linearization of \eqref{eq.Hart} around standing waves, one can apply the classification results from \cite{GS}  and deduce that linearized orbital stability holds for $ p \in  ({5}/{3} , {7}/{3}),$ while linearized orbital instability is fulfilled for $p \in [{7}/{3},5)$.
The notion of orbital stability and the verification that the nonlinear evolution based on \eqref{eq.Hart} is well-defined and gives orbitally stable dynamics for $p \in ({5}/{3} ,  {7}/{3}),$ depend essentially on the local uniqueness of standing waves. More precisely, the standing waves are related to
 the minimization problem
\begin{equation}\label{eq.tdm1}
    \mathcal{E}_\sigma = \inf_{u \in H^1 ,\ \|u\|_{L^2}^2 = \sigma} E_p(u).
\end{equation}
 Here and below
 \begin{equation}\label{eq.3}
    E_{p}(u)=  \frac{1}{2} \|\nabla u\|^2_{L^2}   - \frac{1}{2p} D(|u|^p,|u|^p),
\end{equation}
where
\begin{equation}\label{eq.da1}
D(|u|^p,|u|^p) = \langle I(|u|^p), |u|^p \rangle_{L^2} = \left\| (-\Delta)^{-1/2} |u|^p \right\|_{L^2}^2.
\end{equation}

Any minimizer of \eqref{eq.3} satisfies the Pohozaev identity
$$  \frac{\|\nabla u\|^2}{3p-5}= \frac{D(|u|^p,|u|^p)}{2p}$$
and it is  a solution to the  Euler--Lagrange equation
\begin{equation}\label{Ph1}
   -\Delta u + \omega u = I(|u|^p)|u|^{p-2} u,
\end{equation}
where $\omega>0$ is the Lagrange multiplier.
Then we can write the following  Pohozaev normalization conditions
 \begin{equation}\label{eq.poh2m}
   \frac{\omega\| u\|^2}{\beta} = \frac{\|\nabla u\|^2}{\gamma}= \frac{D(|u|^p,|u|^p)}{p} = k_\mathcal{E},
\end{equation}
  where
  \begin{equation}\label{eq.bg1}
    \beta = \frac{5-p}{2}, \ \gamma = \frac{3p-5}{2} = p-\beta.
  \end{equation}

We start with the following simple property.

\begin{lem} \label{l.11}
Assume
 $p \in (5/3 , 7/3 )$ and $u$ is a minimizer of \eqref{eq.tdm1}.  Then we have the following conditions:
 \begin{itemize}
   \item $u$ satisfies the Euler--Lagrange equation  \eqref{Ph1} with
   \begin{equation}\label{om1}
   \omega = \frac{2\beta}{\gamma - 1}  \ \frac{\mathcal{E}_\sigma}{\sigma};
 \end{equation}
    \item we have the Pohozaev normalization conditions \eqref{eq.poh2m} with

      \begin{equation}\label{om1A}
     k_\mathcal{E} = \frac{2 \mathcal{E}_\sigma}{\gamma -1}.
    \end{equation}
 \end{itemize}
 \end{lem}

We introduce the space
$$
H^1_{rad}=\{u\in H^1(\R^3); u(x)=u(|x|)\}
$$
and state our main result, which treats the local uniqueness of minimizers $Q$ of \eqref{eq.tdm1}.

\begin{thm} \label{Th2} Assume  $2 \leq p < 7/3.$ Then  one can find $\varepsilon > 0$ so that for any two radial positive minimizers $Q_1, Q_2 \in H^1_{rad}$   of \eqref{eq.tdm1}, satisfying
$$ \|Q_1 - Q_2 \|_{H^1_{rad}} \leq \varepsilon ,$$
we have $Q_1=Q_2.$
\end{thm}

The classical case $p=2$  has been studied in \cite{Lieb}, the approach is based on shooting method and the fact that the Riesz potential behaves like
\begin{equation}\label{eq.tr1}
   I(|u|^2)(x) = \frac{\|u\|^2_{L^2}}{4\pi |x|} + o\left(|x|^{-1} \right), \ \ x \to \infty
\end{equation}
so that Pohozaev normalization conditions \eqref{eq.poh2m} in this case become
$$
   \frac{\omega\| u\|^2}{3} = \|\nabla u\|^2 = \frac{D(|u|^2,|u|^2)}{4}.
$$
Indeed, taking any two solutions $u_1,u_2$, we use the previous normalization conditions and from \eqref{eq.tr1} we deduce
$$ I(|u_1|^2)(x) - I(|u_2|^2)(x) = o\left(|x|^{-1} \right), \ \ x \to \infty$$
and this gives the possibility to apply Sturm argument and follow shooting method to deduce  uniqueness.
If $p\neq 2, $
then \eqref{eq.tr1} becomes
$$
 I(|u|^p)(x) = \frac{1}{4\pi} \frac{\|u\|_{L^p}^p}{ |x|} + o\left(|x|^{-1} \right), \ \ x \to \infty $$
 and obviously we loose  the control on the asymptotics of Riesz potential at infinity, since in this case the $L^p$ norm is not presented in Pohozaev normalization conditions \eqref{eq.poh2m}.

There are different method to prove the uniqueness of positive radial minimizes of nonlinear elliptic equations with local type nonlinearities. The method of McLeod and Serin \cite{MS99} and the subsequent refinements due to Kwong \cite{K89} are also based on Sturm oscillation argument and therefore they work effectively for local type nonlinearities. In our case the nonlinearities involve the nonlocal Riesz potential and consequently we have met essential difficulties to follow this strategy.

 Alternative method to show uniqueness of minimizer for Weinstein functionals have been proposed in \cite{CGN07} for the case of local type nonlinearity by studying
 $$\frac{\|u\|_{L^2}^{5-p}\|\nabla u\|_{L^2}^{3p-5}}{D(|u|^p,|u|^p)}.$$
 Performing the substitution of $u$ by $Q+\varepsilon h$ and making a Taylor expansion of the above quotient near $\varepsilon =0$, one can reduce the local existence result to the proof that the operator
 \begin{equation*}
 L_{+}=-\Delta+\omega-pI( Q^{p-1}\cdot)Q^{p-1}-(p-1)I(  Q^{p})Q^{p-2},
\end{equation*}
has a unique negative eigenvalue and a kernel of dimension not greater than $2$.
However, the lack of Sturm comparison argument for nonlocal ODE causes essential difficulties to show the non-degeneracy of $L_+,$ i.e. to check that the kernel of $L_+$ on $H^1_{rad}$ is trivial.
Our approach to obtain the local uniqueness of the minimizer might allow degeneracy of $L_+,$ but the local uniqueness is based on the appropriate analytic continuation $K(z)$ of the function
$$ K : \varepsilon \to E_p\left(\sqrt{\sigma} \frac{Q+ \varepsilon h}{\|Q+ \varepsilon h\|_{L^2}}\right),$$
where $h \in H^1_{rad}$ is a nontrivial element in the kernel of $L_+.$
The crucial point is to show the identity $K(z) = K(0)$ for $z $ in the domain of analyticity of $K(z)$ and to find a suitable curve $z=z(R),$ $R>0$ in this domain so that
$$ \lim_{R \to \infty} K(z(R)) = E_p( \sqrt{\sigma}  \ h).$$

Another question we shall treat in this work is  the characterization of the optimal  constant $C_*$ in the Gagliardo--Nirenberg  inequality
\begin{equation}\label{eq.GNq1}
    D(|u|^p,|u|^p)  \leq C_* \|u\|_{L^2}^{5-p}\|\nabla u\|_{L^2}^{3p-5}.
\end{equation}
Choosing $C_*>0$ to be the best constant in this inequality, we consider the minimization problem
\begin{alignat}{1}
 \label{eq.tdm1a}
 & \mathcal{F}_{\sigma} = \inf_{u \in  H^1, \|u\|^2_{L^2}=\sigma} F_p(u),
\end{alignat}
where
  \begin{equation}\label{eq.GNq2}
   F_p(u) = \|u\|_{L^2}^{5-p}\|\nabla u\|_{L^2}^{3p-5}  - \frac{1}{C_*} D(|u|^p,|u|^p) .
\end{equation}

We focus our interest to
show (at least for $5/3 < p < 7/3$) that the minimizers of \eqref{eq.tdm1} are minimizers of \eqref{eq.tdm1a}.
To give an answer to this question we start with some properties of the minimizers of \eqref{eq.tdm1a}. More precisely, we have the following result.

 \begin{lem} \label{l.fm1}
 Assume $\sigma >0$ and $\omega >0,$ defined by
 \begin{equation}\label{eq.rr1}
  \omega^{1-\gamma} = \frac{C_*}{p} \frac{\gamma^\gamma}{\beta^{\gamma-1}} \sigma^{p-1}.
\end{equation}
If $u$ is a minimizer of \eqref{eq.tdm1a}, then the following conditions are equivalent:
 \begin{description}
   \item[i)]
   \begin{equation}\label{eq.pp1}
    \frac{\|\nabla u\|^2}{\gamma} =  \frac{D(|u|^p,|u|^p)}{p};
 \end{equation}
   \item[ii)]
    \begin{equation}\label{eq.pp2}
    \frac{\omega \sigma}{\beta} =  \frac{D(|u|^p,|u|^p)}{p};
 \end{equation}
   \item[iii)] $u$ is a solution to the Euler--Lagrange equation \eqref{Ph1}.
 \end{description}
 \end{lem}

\begin{defi}
We shall say that the pair $(\sigma,\omega)$ is admissible for  the problem \eqref{eq.tdm1} if the relation  \eqref{om1} is fulfilled.

Similarly, we shall say that the pair $(\sigma,\omega)$ is admissible for  \eqref{eq.tdm1a} if \eqref{eq.rr1} holds.
\end{defi}

Now we are ready to give an answer to the question about the link between the two minimizers.
\begin{thm} \label{l.eq1}   Assume  $p \in (5/3, 7/3).$ Then the following conditions are equivalent:
\begin{description}
  \item[a)] $(\sigma, \omega)$ is admissible pair for  \eqref{eq.tdm1} and $u$ is a minimizer of \eqref{eq.tdm1};
  \item[b)] $(\sigma, \omega)$ is admissible pair for  \eqref{eq.tdm1a} and $u$ is a minimizer of \eqref{eq.tdm1a}.
\end{description}
\end{thm}

\subsection{Properties of  $\mathcal{F}_\sigma$, $\mathcal{E}_\sigma$ and the link among them}

We deal first with the Proof of Lemma \ref{l.11}. Namely we have the following

 \begin{proof}[Proof of Lemma \ref{l.11}]
 It is easy to see, by calculating the first variation of the functional \eqref{eq.tdm1}, that  any non-negative minimizer $Q=Q_\sigma\in H^1_{rad}$ of \eqref{eq.3} satisfies the Euler--Lagrange equation
\begin{align}\label{eq.SNb1m}
&-\Delta Q +\omega Q=I(Q^p)Q^{p-1} ,
  \end{align}
where $\omega = \omega(\sigma)$ is the Lagrange multiplier. In addition we have  also the classical Pohozaev relations
\begin{equation}\label{eq.IPoh}
   \|\nabla Q\|^2_{L^2} + \omega\| Q\|^2_{L^2}- D(|Q|^p,|Q|^p) = 0,
\end{equation}
\begin{equation}\label{eq.PI6}
    \left. \frac{d}{d R} \left( E_p \left(R^{3/2}Q(R x) \right) \right) \right|_{R = 1}  =0
\end{equation}
and
\begin{equation}\label{eq.PI6a}
 E_p(Q) =\mathcal{ E}_\sigma.
  \end{equation}

Combining the relations \eqref{eq.IPoh} and \eqref{eq.PI6}, and taking into account that $\|Q\|_{L^2}^2 = \sigma$, we can represent these relations as the following system
\begin{alignat}{2}\label{eq.PI7} \nonumber
    & \|\nabla Q\|^2_{L^2} + \omega\sigma- D(|Q|^p,|Q|^p) = 0, \\ \nonumber
    & \|\nabla Q\|^2_{L^2} - \frac{3p-5}{2p} D(|Q|^p,|Q|^p) = 0,
    \\
    & \frac{1}{2}\|\nabla Q\|^2_{L^2} - \frac{1}{2p} D(|Q|^p,|Q|^p)  = \mathcal{E}_\sigma.
\end{alignat}

By solving these identities and using the notations \eqref{eq.bg1}, we achieve
\begin{eqnarray}\label{eq.PI11}
   \nonumber D(|Q|^p,|Q|^p) = \frac{\sigma\omega}{\beta}p, \ \\
   \nonumber \|\nabla Q\|^2_{L^2} = \frac{\sigma\omega}{\beta}\gamma, \\
   \mathcal{E}_\sigma=\frac{\sigma\omega}{2\beta}(\gamma-1).
\end{eqnarray}

It is clear now that $\omega>0.$
Then, rearranging the last identity in \eqref{eq.PI11} above, we arrive at \eqref{om1}. Furthermore, the equality \eqref{om1A} is a straightforward consequence of the first two identities in \eqref{eq.PI11}. The proof of the Lemma is now complete.
\end{proof}

\section{Proof of  Theorem \ref{Th2}  }

Our goal is to show the local uniqueness of the minimizer $Q,$
associated to   the minimization problem
$$ \mathcal{E}_\sigma = \inf_{u \in H^1, \|u\|_{L^2}^2=\sigma} E_p(u),$$
where $E_p$  is defined in   \eqref{eq.3}.
The first step is to reduce the local uniqueness to the directional local uniqueness. To be more precise, any vector  $u$ on the sphere $\|u\|^2_{L^2} = \sigma$ close to $Q$ can be represented as $u = \sqrt{\sigma} (Q+\varepsilon h)/\|Q+\varepsilon h\|_{L^2}$
with $h \perp Q $ and $\|h\|_{L^2} = 1$. Without loss of generality we can assume
\begin{equation}\label{eq.ii1}
    Q(x) + \varepsilon h(x) > 0 ,
\end{equation}
provided $ \varepsilon \in I,$ where $I$ is a small interval of type $[0,a]$ with sufficiently small $a>0.$

The minimizer $Q$ will be called  locally unique in direction $h$, if
we can find $\varepsilon_0=\varepsilon_0(h)>0$ and an integer $M>1,$ so that
\begin{equation}\label{eq.u1}
   E_p\left(\sqrt{\sigma} \frac{Q+ \varepsilon h}{\|Q+ \varepsilon h\|_{L^2}}\right) - E_p(Q) \gtrsim \varepsilon^M,
\end{equation}
for any  $ \varepsilon \in (0,\varepsilon_0].$ We shall establish the directional local uniqueness in a way that $\varepsilon_0(h)>0$
will be a continuous function when $h$ is restricted to 2-dimensional subspace.
We argue by contradiction. If the minimizer $Q$ is not unique,
then we can find  sequences $ \varepsilon_k \searrow 0,$ $h_k \perp Q$ so that
$$Q_k = \sqrt{\sigma}(Q+\varepsilon_k h_k)/\|Q+\varepsilon_k h_k\|_{L^2}$$  is a solution to
$$ (\omega-\Delta)Q_k = I(|Q_k|^p) |Q_k|^{p-2}Q_k.$$
Rewriting this equation as
$$ Q_k = (\omega-\Delta)^{-1}I(|Q_k|^p) |Q_k|^{p-2}Q_k $$
and taking the limit $\varepsilon_k \searrow 0,$ we obtain
$$ h_k -\Phi(h_k) \to 0$$
in $L^2$ with
$$\Phi (h) = (\omega-\Delta)^{-1}\left((p-1)I(Q^p)Q^{p-2}h + p I(Q^{p-1}h) Q^{p-1}\right) $$ being a compact operator in $H^1.$ Then $h_k$ is convergent on $L^2$ to $h$ and  satisfies $L_+(h) =0.$ Therefore, it remains to show the directional local uniqueness for $h$ in the kernel of $L_+.$ Note that this kernel has dimension at most 2 due to Lemma \ref{l.22}. If for $h \in \Ker L_+$ the property
\eqref{eq.u1} is not true, then we can find decreasing sequence $\varepsilon_k \to 0,$ such that
\begin{equation}\label{eq.u1a}
    0 \leq E_p\left(\sqrt{\sigma} \frac{Q+ \varepsilon_k h}{\|Q+ \varepsilon_k h\|_{L^2}}\right) - E_p(Q) \lesssim \varepsilon_k^M ,
\end{equation}
for any $M>1$. However, for any smooth function $ F,$ such that there exists a sequence $\varepsilon_k\to 0,$ with the property
$|F(\varepsilon_k)-F(0)| \lesssim |\varepsilon_k|^2,$ one can assert that $F^\prime(0)=0.$ In a similar way, if there exists an integer $M>1$ and a sequence $\varepsilon_k\to 0,$ such that
$|F(\varepsilon_k)-F(0)| \lesssim |\varepsilon_k|^M,$ then all derivatives of $F$ up to order $M-1$ are identically zero.
Therefore, \eqref{eq.u1a}  implies that all derivatives of the function
$$ K : \varepsilon \to E_p\left(\sqrt{\sigma} \frac{Q+ \varepsilon h}{\|Q+ \varepsilon h\|_{L^2}}\right) $$ at $\varepsilon =0$ are identically zero.
We have the relation
$$ E_p\left(\sqrt{\sigma} \frac{Q+ \varepsilon h}{\|Q+ \varepsilon h\|_{L^2}}\right)  = $$ $$ = \frac{\sigma (\|\nabla Q\|^2 + \varepsilon^2 \|\nabla h\|^2)}{2(\sigma + \varepsilon^2)} - \frac{\sigma^p D(|Q+\varepsilon h|^p,|Q+\varepsilon h|^p)}{2p( \sigma+\varepsilon^2)^p}.$$
This function can be extended as analytic function
$$ K : z \to E_p\left(\sqrt{\sigma} \frac{Q+ z h}{\sqrt{\|Q\|^2_{L^2}+ z^2 }}\right), \ \ z\in \C,$$
in a small neighborhood, say  $ |z| < 4\delta$ with $\delta >0$ sufficiently small to be chosen later on.
We obviously have the analyticity of
$$ z \to \frac{\sigma (\|\nabla Q\|^2 + z^2 \|\nabla h\|^2)}{2(\sigma + z^2)}$$ near $z=0.$ More delicate is the analyticity of the map
$$ z \to D((Q+z h)^p,(Q+z h)^p).$$
In this case, we can apply Proposition \ref{pr.s1} and use the estimate $$|h(r)|/Q(r) \leq C.$$ Then $ \Re (1 +  z h(r)/Q(r)) > 1/2$ for $|z|$ small and the function
$$  z \to  \left(1 + z \frac{h(r)}{Q(r)}\right)^p  $$
is analytic near the origin, say
$\{ |z| < 4\delta  \}.$ Therefore, $$ z \to \int_{\R^3}\int_{\R^3} \left(1 + z \frac{h(|x|)}{Q(|x|)}\right)^p
\left(1 + z \frac{h(|y|)}{Q(|y|)}\right)^p \frac{Q(x)^p Q(y)^p dx dy}{|x-y|}$$ is analytic in the same disk. Moreover, setting
 $$ w=w(z) =  1 + z \frac{h(r)}{Q(r)},$$
   we have on the line $\{ \mathrm{Re} z = \mathrm{Im} z \}$ the property
$$ \mathrm{Re}  \ w(z)  = 1 + \mathrm{Re }z \frac{h(r)}{Q(r)}  = 1 + \mathrm{Im } z \frac{h(r)}{Q(r)} = 1+ \mathrm{Im}\   w(z).$$
Since the  principal value of $\mathrm{Log} \ w$ can be defined on the line $\mathrm{Re}  w(z)= 1+ \mathrm{Im}  w(z)$ as well as on  its  small neighborhood
$$ \Lambda_{\delta}=\{ |\mathrm{Re} z  - \mathrm{Im} z | < \delta, \mathrm{Re} z >0, \mathrm{Im} z > 0 \}. $$
Indeed,
we have
$$ \mathrm{Re} w -  \mathrm{Im }w = 1 + (\mathrm{Re} z - \mathrm{Im} z) \frac{h(r)}{Q(r)} \geq 1 - \delta \frac{|h(r)|}{Q(r)} \geq 1-\delta C.$$
In conclusion we have  analytic extension of
$$ z \to D((Q+z h)^p,(Q+z h)^p)$$
in the domain
$$\Omega_{\delta} = \{ |z| \leq 4 \delta \} \cup \Lambda_{\delta}.$$

The assumption \eqref{eq.u1a} means that all derivatives of $K(z)$ at $z=0$ are identically zero, so the function $K(z)$ is a constant
\begin{equation} \label{an1}
    K(z) = K(0).
\end{equation}

Our next step is to show that $K(z)$ can be extended as analytic function in
$ \Omega_\delta .$
Indeed, we can  show the analyticity of $\Arg (\sigma + z^2)$ on $\Omega_\delta.$ For $|z| < 4\delta$ and $\delta < \sqrt{\sigma}/8$ one has $\Re (\sigma+z^2)>3\sigma/4.$ For $|z|> 4\delta$ and $z \in \Lambda_\delta$ it is easy to see that $\Re z>2\delta$, then we have $$\Im (\sigma+z^2) = 2 (\Re z)(\Im z)= 2(\Re z)^2 + 2\Re z(\Im z - \Re z) >  $$ $$ > 2(\Re z)^2 - 2\Re z \delta = 2\Re z ( \Re z - \delta)  > 4  \delta ( 2 \delta - \delta)= 4 \delta^2 .$$
This shows that we can extend $K(z)$ as analytic function in the domain $\Omega_\delta,$ so we can extend  the relation \eqref{an1} in the whole  $\Omega_\delta.$

\begin{figure}[h!]
 \centering
       \includegraphics[totalheight=6.5cm]{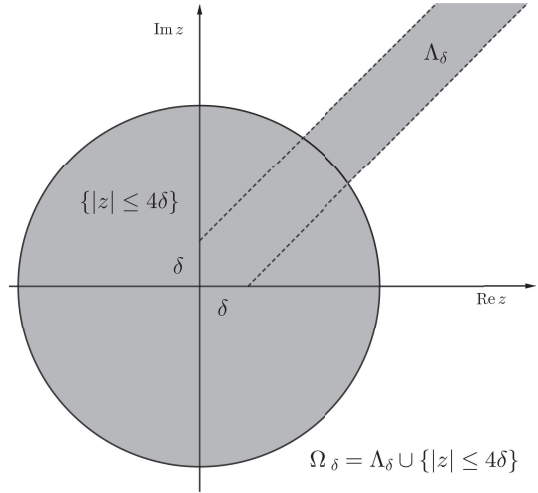}
 \caption{The domain of analyticity of $K(z)$, here depicted by the shaded region $ \Omega_\delta$.}
 \label{fig:Pic1}
\end{figure}

Choosing $ z(R) = R + i R  $ with $R \to \infty,$ we can use the relation
$$ \frac{1+z(R)h(|x|)/Q(|x|)}{\sqrt{\sigma+z(R)^2}} \to \frac{h(|x|)}{Q(|x|)},$$
combined with Lebesgue dominated convergence theorem to conclude that
$$ \lim_{R \to \infty} K(z(R)) = E_p(\sqrt{\sigma} \  h).$$
The relation
$$ E_p(\sqrt{\sigma} \ h) = K(0) = E_p(Q)$$
shows that $u(|x|) = \sqrt{\sigma} \ h$ is a minimizer of $E_p,$ satisfying the constraint condition $\|u\|^2_{L^2}=\sigma$. Hence the same is true for $|u(|x|)|$ and both of them  satisfy the equation
$$ - \Delta u +\omega u = I(|u|^p) |u|^{p-2} u.$$
Since $h$ is orthogonal to $Q,$ there exists $r_0>0,$ such that $h(r_0)=u(r_0)=0.$ Therefore, we are in position to apply Lemma \ref{sl1} and to conclude that $u(r)=0$
for any $r>0.$ This is an obvious contradiction and shows that for any $ h \in \Ker L_+$,
we can find $\varepsilon_0=\varepsilon_0(h)>0$, $\delta_0=\delta_0(h)>0$ and  an integer $M>1,$ so that \eqref{eq.u1} is fulfilled
for any $ \varepsilon \in (0,\varepsilon_0].$

Recalling that Lemma \ref{l.22}
guarantees that the kernel of $L_+$ has dimension at most 2. Thus, we can show that there exists uniform $\varepsilon_0 >0,$ such that for any $h$ in the kernel of $L_+$ the property
\eqref{eq.u1} is fulfilled
for $\varepsilon \in (0,\varepsilon_0].$

The last assertion can be verified by assuming the opposite and finding a sequence $h_k \to h^* \perp Q$, $ \|h^*\|_{L^2}=1$ such that the
function
$$ z \to K^*(z) =  E_p\left(\sqrt{\sigma} \frac{Q+ z h^*}{\|Q+ z h^*\|_{L^2}}\right) $$
has all derivatives equal to zero at the origin. As above, the analytic extension of $K^*(z)$ in $\Omega_\delta$ shows that $h^* = 0$ and this contradiction completes the proof.

\section{Characterization of Gagliardo--Nirenberg optimal constant}

We start this section by the simple observation that for any $\sigma >0,$
the  minimization problem
$$
   \mathcal{F}_{\sigma} = \inf_{u \in  H^1, \|u\|^2_{L^2}=\sigma} F_p(u).
$$
 has infimum $\mathcal{F}_{\sigma}=0.$

 \begin{proof}[Proof of Lemma \ref{l.fm1}] The Pohozaev conditions for the minimizers of \eqref{eq.tdm1a} have the form
 \begin{equation}\label{eq.le1}
  \frac{\|\nabla u\|^2}{\gamma} = \frac{\omega \sigma}{\beta} = \frac{D(|u|^p,|u|^p)}{p} .
 \end{equation}
 The assumption that  $u$ is a minimizer of \eqref{eq.tdm1a} has the meaning that the Gagliardo--Nirenberg equality
 $$ D(|u|^p,|u|^p) = C_*\|\nabla u\|^{2\gamma} \sigma^\beta$$
 holds.
Moreover, for any  $ \sigma >0$, the Euler--Lagrange equation for minimizers of $\mathcal{F}_{\sigma}$ is
 $$ -\Delta u + \Lambda u =  I(|u|^p) |u|^{p-2}u.$$
First we note that {\bf iii)} is equivalent to  \eqref{eq.le1} and therefore $\Lambda=\omega$. Moreover, Gagliardo--Nirenberg equality combined with  \eqref{eq.le1} give \eqref{eq.rr1}, so $(\sigma,\omega)$ is admissible pair for \eqref{eq.tdm1a} and we have {\bf iii)} $ \Longrightarrow$ {\bf i)} and {\bf ii)}.
From Gagliardo--Nirenberg equality, \eqref{eq.rr1} and {\bf i)} imply
 $$ D(|u|^p,|u|^p) = C_* \|\nabla u \|^{2\gamma} \sigma^\beta = \frac{p \|\nabla u \|^2}{\gamma},$$
 then
 $$  \|\nabla u \|^{2(1-\gamma)} = \frac{C_* \gamma}{p}  \sigma^\beta.$$
 Now \eqref{eq.rr1} can be rewritten as
 $$  \frac{C_* \gamma}{p}  \sigma^\beta = \left(\frac{\gamma \omega \sigma}{\beta} \right)^{1-\gamma}$$
 and we arrive at
  \eqref{eq.le1} so we conclude that {\bf i)} $ \Longrightarrow$ {\bf iii)}. In a similar way we check
  {\bf ii)} $ \Longrightarrow$ {\bf iii)}.
This completes the proof.
\end{proof}

Our next step is to connect the minimizers of $\mathcal{F}_\sigma$ with the  minimization problem
\begin{equation}\label{eq.Mp12}
    \mathcal{E}_\sigma = \inf_{u \in H^1 ,\ \|u\|_{L^2}^2 = \sigma} E_p(u).
\end{equation}

\begin{proof}[Proof of Theorem \ref{l.eq1}]
  {\bf a)}$ \Longrightarrow$  {\bf b)}:   If $(\sigma, \omega)$ is admissible pair for  \eqref{eq.tdm1}, then we have \eqref{eq.rr1}.

The plan is to assume that $u$ is a minimizer of \eqref{eq.tdm1} and to prove
\begin{equation}\label{eq.GNe1}
    D(|u|^p,|u|^p)  = C_* \|u\|_{L^2}^{5-p}\|\nabla u\|_{L^2}^{3p-5}.
\end{equation}
For the purpose we shall assume that
\begin{equation}\label{eq.ii1}
  D(|u|^p,|u|^p)  < C_* \|u\|_{L^2}^{5-p}\|\nabla u\|_{L^2}^{3p-5}.
\end{equation}
and we shall arrive at contradiction.
From \eqref{eq.ii1} we have the inequality
$$ E_p(u) > \frac{1}{2} \|\nabla u\|^2_{L^2} - \frac{C_* \sigma^\beta}{2p} \ \|\nabla u\|_{L^2}^{2\gamma}, $$
with $\beta, \gamma$ defined in \eqref{eq.bg1}.
The right hand side suggests us to consider the function
\begin{equation}\label{eq.vp1}
  \varphi(s) = \varphi_\sigma(s) = \frac{s}{2}  - \frac{C_* \sigma^\beta}{2p} \ s^{\gamma}, \ \ s \geq 0
\end{equation}
and obviously we have then
\begin{equation}\label{eq.re1}
   E_p(u) > \varphi \left( \|\nabla u\|^2_{L^2}\right) \geq \min_{s \geq 0} \varphi(s) = \varphi(s_*),
\end{equation}
with $s_*$ being the unique solution to the equation
$$ s_* = \frac{\gamma}{p} C_* \sigma^\beta s_*^\gamma.$$
Further we take any minimizer $v$ of \eqref{eq.tdm1a} and then we know that $\|v\|_{L^2}^2 =\sigma$ and
\begin{equation}\label{eq.GNn1}
    D(|v|^p,|v|^p)  = C_*\sigma^{\beta}\|\nabla v\|_{L^2}^{2\gamma}.
\end{equation}
Moreover, any rescaled function
$$ v_\mu(x) = \mu^{3/2} v(\mu x) $$
generated by $v$ preserves the $L^2$ norm and the Gagliardo--Nirenberg equality \eqref{eq.GNn1}. Now we choose $\mu$ in such a way so that
$$ \|\nabla v_\mu\|_{L^2}^2 = \mu^2 \|\nabla v\|_{L^2}^2 = s_*.$$
Then we have
$$ \varphi(s_*) = \varphi \left( \|\nabla v_\mu\|^2_{L^2}\right) = E_p(v_\mu)$$
and we arrive at
$$ \mathcal{E}_\sigma = E_p(u) > E_p(v_\mu),$$
with $\|v_\mu\|_{L^2}^2 =\sigma$ and this is clearly in contradiction with the fact that $u$ is a minimizer of \eqref{eq.tdm1}.

{\bf b)}$ \Longrightarrow$  {\bf a)}: We assume that $v$ satisfies Gagliardo--Nirenberg equality \eqref{eq.GNe1}, $\|v\|_{L^2}^2=\sigma$ and we have Pohozaev normalization conditions
\begin{equation}\label{eq.poh23}
   \frac{\omega \sigma}{\beta} = \frac{\|\nabla v\|^2}{\gamma}= \frac{D(|v|^p,|v|^p)}{p} ,
\end{equation}
as stated in Lemma \ref{l.fm1}. We shall use the properties of the function $\varphi_\sigma(s)$ defined in \eqref{eq.vp1}. As before, we choose $s_*$ to be  the point of minimum of this function.
Next,  we choose the parameter $\mu>0$ so that
$v_\mu(x) = \mu^{3/2} v(\mu x)$ satisfies
$$ \| \nabla v_\mu \|_{L^2}^2 = s_*.$$ Then $v_\mu$ satisfies the Gagliardo--Nirenberg equality and hence
\begin{equation}\label{tl1}
  E_p(v_\mu) = \varphi_\sigma(s_*).
\end{equation}
It is not difficult to show that
\begin{equation}\label{tl2}
  \varphi_\sigma(s_*) = \mathcal{E}_\sigma .
\end{equation}
Indeed, the identity \eqref{tl1} implies
$$ \varphi_\sigma(s_*) =  E_p(v_\mu) \geq \mathcal{E}_\sigma.$$
 If we take any minimizer $u$ of \eqref{eq.tdm1} we know from the step {\bf a)}$ \Longrightarrow$  {\bf b)}, that $u$ satisfies the Gagliardo--Nirenberg equality
$$  \mathcal{E}_\sigma =E_p(u) = \varphi_\sigma ( \|\nabla u\|_{L^2}^2 ) \geq \varphi_\sigma(s_*).$$
Therefore, we arrive at \eqref{tl2} and the identities
$$  E_p(v_\mu) = \varphi_\sigma(s_*) = \mathcal{E}_\sigma$$
guarantee that $v_\mu$ is a minimizer of  \eqref{eq.tdm1}, so we can find its Lagrange multiplier $\omega(\mu)$ such that
$$  \frac{\omega(\mu) \sigma}{\beta} = \frac{\|\nabla v_\mu\|^2}{\gamma}= \frac{D(|v_\mu|^p,|v_\mu|^p)}{p}.$$

On the other hand $v$ satisfies \eqref{eq.poh23}
and the simple rescaling relations
$  \|\nabla v_\mu\|^2 = \mu^2 \|\nabla v\|^2,$ $ D(|v_\mu|^p,|v_\mu|^p) = \mu^{2\gamma} D(|v|^p,|v|^p)$ show immediately that $\mu=1$ and $\omega(\mu)=\omega.$

\end{proof}

\begin{cor} Assume $p \in (5/3, 7/3).$
If $(\sigma, \omega)$ is admissible pair for  \eqref{eq.tdm1}, then we have the following relation between $\mathcal{E}_\sigma < 0$ and the best Gagliardo--Nirenberg constant $C_*$
\begin{equation}\label{eq.gne1}
   C_* = \frac{p}{\gamma^\gamma} \left(\frac{2}{1-\gamma}\right) ^{1-\gamma} \ \frac{|\mathcal{E}_\sigma|^{1-\gamma}}{\sigma^{p-\gamma}}.
\end{equation}
\end{cor}

\section{Asymptotics at infinity}

The vector $h \perp Q$ in the kernel of $L_+$ satisfies the equations (here for simplicity we take  $\omega=1$)
\begin{eqnarray} \label{eq.As1}
 & & (1 - \Delta)h = pB Q^{p-1} + (p-1) A Q^{p-2}h, \\ \nonumber
  & &-\Delta B = Q^{p-1}h.
\end{eqnarray}
Note that the positive radial  ground state  $Q$ satisfies the system
\begin{eqnarray} \label{eq.As2}
 & & (1 - \Delta)Q = A Q^{p-1} , \\ \nonumber
  & &-\Delta A = Q^{p}.
\end{eqnarray}
By using the arguments in \cite{St77}, we have the following asymptotic expansions of $Q $ and $ A$ as $r \to \infty$
\begin{equation} \label{eq.As3}
    Q(r) = \frac{e^{-r}}{r} \left( c_0 +  O \left(\frac{1}{r} \right) \right),
\end{equation}
\begin{equation} \label{eq.As4}
    A(r) = \frac{1}{r} \left( d_0 +  O \left(\frac{1}{r} \right) \right),
\end{equation}
with $c_0 , d_0 >0.$
\begin{prop} \label{pr.s1}
If $h$ is a nontrivial solution to \eqref{eq.As1}, then there exists $r_* >0$, so that $h(r) \neq 0$ for $r > r_*$ and the following estimate holds
\begin{equation} \label{eq.As5i}
    |h(r)| \lesssim  \frac{e^{-r}}{r} , \ \
\end{equation}
\end{prop}
\begin{proof}
The starting point are the following asymptotics (verified in a similar way to \eqref{eq.As3} and \eqref{eq.As4})
\begin{equation} \label{eq.As5}
    h(r) = \frac{e^{-r}}{r} \left( c_1 +  O \left(\frac{1}{r} \right) \right),
\end{equation}
$$ c_1 = \int_{\R^3}  pB(|y|) Q^{p-1}(|y|) + (p-1) A(|y|) Q^{p-2}(|y|)h(|y|) dy, $$
\begin{equation} \label{eq.As6}
    B(r) = \frac{1}{r} \left( d_1 +  O \left(\frac{1}{r} \right) \right),
\end{equation}
\begin{equation} \label{eq.As7}
    B^\prime(r) = -\frac{1}{r^2} \left( d_1 +  O \left(\frac{1}{r} \right) \right),
\end{equation}
$$ d_1 = \int_{\R^3}  Q^{p-1}(|y|)h(|y|) dy. $$
If $c_1\neq 0,$ then the assertion of the Proposition follows. If $c_1=0,$ then we can show that $d_1=0.$
Indeed, if $d_1 \neq 0,$ then without loss of generality we can assume $d_1>0,$ so we can find a sufficiently large $r_0,$ so that  $B(r) >0$ and $ V(r) = A(r)Q^{p-2}(r) < 1,$ for  $r >r_0.$ Let us assume that
$h(r)=0$ has two roots $r_2 > r_1 > r_0$ and
$\min_{[r_1,r_2]} h(r) = h(\tilde{r}) <0.$ Then the maximum principle for the equation
\begin{equation}\label{eq.As8}
(1-\Delta-V(r))h(r) = B(r) Q^{p-1}(r)
\end{equation}
in the interval $[r_1,r_2]$ leads to a contradiction. Indeed, in the point $\tilde{r}$ of the negative minimum of $h$ we have
$\Delta h(\tilde{r}) \geq 0 $, then
$$  (1-\Delta-V(\tilde{r}))h(\tilde{r}) \leq 0.$$
This obviously contradicts the positiveness of the right hand side in \eqref{eq.As8}.
The contradiction shows that $c_1=d_1=0.$
Then we can perform the substitution
$ h(r) = e^{-r}g(r)/r$ into \eqref{eq.As1} and
deduce the equations
\begin{eqnarray} \label{eq.As9}
 & & -g^{\prime\prime}(r) + 2 g^\prime(r)= F_1, \\ \nonumber
  & &- B^{\prime\prime}(r) - \frac{2}{r} B^\prime(r) = F_2,
\end{eqnarray}
with
$$ F_1(r)= pr e^{r} B(r) Q^{p-1}(r) + (p-1) A(r) Q^{p-2}(r) g(r), $$
$$ F_2(r) = Q^{p-1}(r) \frac{e^{-r}g(r)}{r}.$$
The asymptotic expansions \eqref{eq.As3}, \eqref{eq.As4} as well the ones in
\eqref{eq.As6} and \eqref{eq.As7} with $c_1=d_1=0$ imply the estimates
$$ |F_1(r)| \lesssim \frac{e^{-(p-2)r}}{r^{p-2} }|B(r)|+ \frac{e^{-(p-2)r}}{r^{p-1} }|g(r)|,$$
$$ |F_2(r)| \lesssim \frac{e^{-pr}}{r^{p} }|g(r)|.$$
Integrating the equations \eqref{eq.As9} from $r$ to $\infty,$ we find
$$ g(r) \lesssim \int_r^\infty \frac{e^{-(p-2)s}}{s^{p-3} }|B(s)|+ \frac{e^{-(p-2)s}}{s^{p-2} }|g(s)| ds,$$
$$ B(r) \lesssim \int_r^\infty \frac{e^{-ps}}{s^{p-1} }|g(s)| ds .$$
To this end we can use the following Lemma with $\psi(r) = |g(r)|+ |B(r)|.$

\begin{lem} (see Lemma 4.1 in \cite{GTV}) \label{l.g1}
If $\varepsilon > 0,$ $\psi(r) \in C(1,\infty)$ is a non negative function satisfying
\begin{equation}\label{eq.A1}
    \psi(r) \leq C, \ \ \forall r > 1
\end{equation}
and
$$ \psi(r) \leq C \int_r^\infty \frac{\psi(s) ds}{s^{1+\varepsilon}}, \ \ \forall r> 1,$$
then $\psi(r) =0$ for $r >1.$
\end{lem}
An application of this Lemma guarantees that $h(r)=0$ and this contradiction completes the proof.
\end{proof}

\section{Simple ODE lemmas}

In case $u(|x|)$ is a radial $C^1$-solution of the equation
\begin{equation} \label{ode1}
(\omega - \Delta) u = V(|u|) u,
\end{equation}
with $V(|u|)(|x|)$
being a continuous function in $|x|>0$, we have the following result.
\begin{lem} \label{sl1}
If $u$ and $|u|$ solve \eqref{ode1}, $u \in C^1(0,\infty)$ and there exists $r_0>0,$ such that $u(r_0)=0,$ then $u(r) \equiv 0.$
\end{lem}
\begin{proof}
If $u^\prime(r_0) =0,$ then the Cauchy problem for the ODE \eqref{ode1} implies the assertion. If $u^\prime(r_0) < 0,$ then $|u(r)|$ is not differentiable in $r_0.$
The proof is now completed.
\end{proof}

Next we discuss the dimension of the kernel of $L_+$.

\begin{lem} \label{l.22} If  $2 < p < 7/3,$ then we have
$$ \mathrm{dim } (\mathrm{Ker} L_+) \leq 2. $$
\end{lem}
\begin{proof}
Any positive radial solution $w$ to the equation $L_+ w=0$ is a solution of the ordinary differential equation
$$ -r^{-2} \partial_r ( r^{2} \partial_r w(r)) + \omega w = pI( Q^{p-1}w)Q^{p-1}+(p-1)I(  Q^{p})Q^{p-2}w.$$
Then the couple of $w$ and $ B=I(Q^{p-1}w)$ satisfies the
 system of nonlinear second order differential equations
\begin{alignat}{2}\label{fp.10}
    & w^{\prime\prime}(r) + \frac{2}{r}  w^\prime(r) = \omega w(r) - p B Q^{p-1}-(p-1)I(  Q^{p})Q^{p-2}w , \\ \nonumber
    & B^{\prime\prime}(r) + \frac{2}{r} B^\prime(r)  = -Q^{p-1}w.
\end{alignat}
subject to  initial data
\begin{alignat}{2}\label{fp.12}
    & w(0) = w_0 \neq  0, \ \ B(0) = B_0 \neq 0, \\ \nonumber
    & w^\prime(0) =  0, \ \ B^\prime(0) =  0.
\end{alignat}
The Fuchs--Painleve Theorem
\ref{fpt1} gives the series expansions
\begin{equation}\label{eq.fp.16}
  w(r) = w_0 + \sum_{k=1}^\infty w_{2k} r^{2k}, \ \ B(r) = B_0 + \sum_{k=1}^\infty B_{2k} r^{2k},
\end{equation}
where all coefficients $w_{2k}, B_{2k}, k \geq 1$ can be determined in a unique way by the recurrence relations in terms of the two free initial data
$w_0$ and $B_0$. This completes the proof of the Lemma.
\end{proof}

\section{Appendix: Fuchs--Painleve series expansions of ground states }

The equation
\begin{equation}\label{t:20}
   -\Delta u + E u = I(u^p)u^{p-1}
\end{equation}
 can be rewritten as a system of nonlinear second order differential equations
\begin{alignat}{2}\label{fp.1}
    & Q^{\prime\prime}(r) + \frac{2}{r}  Q^\prime(r) = EQ - A(r) Q^{p-1}, \\ \nonumber
    & A^{\prime\prime}(r) + \frac{2}{r}  A^\prime(r)  = -Q^{p}.
\end{alignat}
Our goal will be to  verify that imposing special initial data
\begin{alignat}{2}\label{fp.2}
    & Q(0) = Q_0 > 0, \ Q^\prime(0) =  0, \  \\ \nonumber
    &  A(0) = A_{0}, \ \ A^{\prime}(0) =  0,
\end{alignat}
we can find unique real analytic (near $r=0$) solution to this Cauchy problem. Then we can consider the following more general problem
\begin{alignat}{2}\label{fp.3}
    & Y^{\prime\prime}(r) + \frac{c}{r}  Y^\prime(r) = F(r,Y), \\ \nonumber
    & Y(0) = Y^\prime(0) = 0,
\end{alignat}
where we have shifted the initial data to zero, but we assume that $F(r,0) \neq 0$ may be nontrivial source term.
To be more precise, here $Y(t) \in C^2([0,1); \R^3 )$ is a vector - valued function, while $F$ satisfies the assumptions
\begin{equation}\label{fp.4}
  \mbox{$F(r,Y)$ is real analytic near $r=0, Y=0$}
\end{equation}
and
\begin{equation}\label{fp.5}
    F(0,0) \neq 0.
\end{equation}
As in Theorem 11.1.1 in \cite{H76} we can state the following Fuchs--Painleve type result

\begin{thms} \label{fpt1}
If the conditions \eqref{fp.4} and \eqref{fp.5} are fulfilled, then the Cauchy problem \eqref{fp.3} has a unique real analytic solution
$$ Y(r) = \sum_{k=2}^\infty Y_k r^k $$ near $r=0.$
\end{thms}

This result applied to the Cauchy problem \eqref{fp.1}, \eqref{fp.2} gives the following series expansions near $r=0$
\begin{equation}\label{eq.fp.6}
  Q(r) = Q_0 + \sum_{k=1}^\infty Q_{2k} r^{2k}, \ \ A(r) = A_0 + \sum_{k=1}^\infty A_{2k} r^{2k}.
\end{equation}

\end{document}